\documentclass[10pt]{amsart}
\usepackage{amssymb,amsmath,txfonts,amsthm}
\usepackage{hyperref}
\usepackage{mathrsfs}
\newcommand\blfootnote[1]{%
  \begingroup
  \renewcommand\thefootnote{}\footnote{#1}%
  \addtocounter{footnote}{-1}%
  \endgroup
}
\newtheorem{theorem}{Theorem}
\newtheorem{prop}{Proposition}
\newtheorem{lemma}{Lemma}
\newtheorem{remark}{Remark}
\newtheorem{claim}{Claim}
\newtheorem{definition}{Definition}
\newtheorem{cor}{Corollary}

\numberwithin{equation}{section}
\numberwithin{theorem}{section}
\numberwithin{definition}{section}
\numberwithin{cor}{section}
\numberwithin{prop}{section}
\numberwithin{remark}{section}
\numberwithin{claim}{section}
\numberwithin{lemma}{section}

\def\Xint#1{\mathchoice
  {\XXint\displaystyle\textstyle{#1}}%
  {\XXint\textstyle\scriptstyle{#1}}%
  {\XXint\scriptstyle\scriptscriptstyle{#1}}%
  {\XXint\scriptscriptstyle\scriptscriptstyle{#1}}%
  \!\int}
\def\XXint#1#2#3{{\setbox0=\hbox{$#1{#2#3}{\int}$}
  \vcenter{\hbox{$#2#3$}}\kern-.5\wd0}}

\def\dashint{\Xint-}

\author{Gang Liu}
\address{Department of Mathematics\\Northwestern University\\Evanston, IL 60208}
\email{gangliu@math.northwestern.edu}
\title[Compactification]{Compactification of certain K\"ahler manifolds with nonnegative Ricci curvature}
\date{}
\begin{document}
\begin{abstract}
We prove compactification theorems for some complete K\"ahler manifolds with nonnegative Ricci curvature. Among other things, we prove that a complete noncompact K\"ahler Ricci flat manifold with maximal volume growth and quadratic curvature decay is a crepant resolution of a normal affine algebraic variety. Furthermore, such affine variety degenerates in two steps to the unique metric tangent cone at infinity.
\end{abstract}
\maketitle

\section{\bf{Introduction}}
\blfootnote{The author was partially supported by NSF grant DMS-1406593, DMS-1709894 and the Alfred P. Sloan fellowship.}
In \cite{[Y1]}, Yau proposed the uniformization conjecture which states that a complete noncompact K\"ahler manifold with positive bisectional curvature is biholomorphic to $\mathbb{C}^n$. In \cite{[L1]}-\cite{[L6]}, we studied the uniformization conjecture and its related problems. One of the main tools is the Gromov-Hausdorff convergence theory developed by Cheeger-Colding \cite{[CC1]}-\cite{[CC4]} and Cheeger-Colding-Tian \cite{[CCT]}.

In this paper, we extend some techniques in \cite{[L1]}-\cite{[L6]} to study the compactification of certain complete K\"ahler manifolds with nonnegative Ricci curvature. 
\begin{definition}\cite{[LW]} \cite{[TY]}
On a K\"ahler manifold $M^n$, we say the bisectional curvature is greater than or equal to $K$ ($BK\geq K$), if
\begin{equation}
\frac{R(X, \overline{X}, Y, \overline{Y})}{||X||^2||Y||^2+|\langle X, \overline{Y}\rangle|^2}\geq K 
\end{equation}
for any two nonzero vectors $X, Y\in T^{1, 0}M$.\end{definition}
The bisectional curvature lower bound condition is weaker than the sectional curvature lower bound, while stronger than the Ricci curvature lower bound. 

The main result is
\begin{theorem}\label{thm1}
Let $(M^n, p)$ ($n\geq 2$) be a complete noncompact K\"ahler manifold with nonnegative Ricci curvature and maximal volume growth.  Let $r(x) = d(x, p)$. Then

(I)
$M$ is biholomorphic to a Zariski open set of a Moishezon manifold, if for some $\epsilon>0$, the bisectional curvature $BK\geq -\frac{C}{r^{2+\epsilon}}$.
If fact, on $M$, the ring of polynomial growth holomorphic functions is finitely generated. 

(II)
If $BK\geq -\frac{C}{r^2}$ and $M$ has a unique tangent cone at infinity, then $M$ is biholomorphic to a Zariski open set of a Moishezon manifold.

(III)
$M$ is quasiprojective, if the Ricci curvature is positive and $|Rm|\leq \frac{C}{r^2}$.

\end{theorem}

Combining part (II) with some argument in \cite{[DS2]} or \cite{[CSW]}, we obtain
\begin{cor}\label{cor1}
Let $M$ be a complete noncompact K\"ahler-Ricci flat manifold with maximal volume growth. Assume the curvature has quadratic decay. Then $M$ is a crepant resolution of a normal affine algebraic variety. Furthermore,
there exist two step degenerations from that affine variety to the unique metric tangent cone of $M$ at infinity.
\end{cor}

\begin{remark}
It is desirable to remove the uniqueness of the tangent cone in part II. 
\end{remark}

\begin{remark}
A conjecture of Yau \cite{[Y2]} states that if a complete Ricci flat K\"ahler manifold has finite topological type, then it can be compactified complex analytically. Corollary \ref{cor1} supports the conjecture, at least in this very special setting.

Another conjecture of Yau (question $71$ in \cite{[Y3]}, page $304$) states that complete noncompact K\"ahler manifolds with positive Ricci curvature is biholomorphic to a Zariski open set of a compact K\"ahler manifold.
Part III of theorem \ref{thm1} supports this conjecture.
\end{remark}

Theorem \ref{thm1} is a generalization of several known results. For instance, theorem \ref{thm1}, part I is a generalization of the following \begin{theorem}\cite{[L2]}
Let $M^n$ be a complete noncompact K\"ahler manifold with nonnegative bisectional curvature and maximal volume growth. Then $M$ is biholomorphic to an affine algebraic variety \footnote{Recently, in \cite{[L6]}, it was proved that $M$ is in fact biholomorphic to $\mathbb{C}^n$.}.
\end{theorem}

\medskip

In a series of papers, R. Conlon and H. Hein \cite{[CH1]}-\cite{[CH3]} systematically studied asymptotically conical Calabi-Yau manifolds. In some sense, corollary \ref{cor1} is a generalization of some of their results.

\medskip
In \cite{[Mo]}, Mok proved the following
\begin{theorem}
Let $M^n$ be a complete noncompact K\"ahler manifold with positive Ricci curvature and maximal volume growth. Assume $|Rm|\leq\frac{C}{r^2}$ and $\int_MRic^n<+\infty$, then $M$ is biholomorphic to a quasi-projective variety.
\end{theorem}

Theorem \ref{thm1}, part III removes the assumption that $\int_MRic^n<+\infty$.

\medskip

The strategy of part I and II is to consider polynomial growth holomorphic functions. We shall follow the argument in \cite{[L3]}. However, there are some differences.

\medskip

1.  In this paper, we construct plurisubharmonic functions by using elliptic method (see proposition \ref{p3}). In \cite{[L3]}, the parabolic method of Ni-Tam \cite{[NT1]} was adopted.

\medskip

2. The original three circle theorem in \cite{[L1]} does not work in this paper. In part I, we just use the extended version in \cite{[L1]} (proposition \ref{p1}). In part II, we apply Donaldson-Sun's three circle theorem \cite{[DS2]} (lemma \ref{l41}).

\medskip

3. In the setting of \cite{[L3]}, polynomial growth holomorphic functions separate points and tangents. This is no longer true in part I and II, due to the possibility of compact subvarieties.

\bigskip

In some sense, part III resembles part I and II.
However, the argument is very different. We basically follow the argument of Mok \cite{[Mo]}. The strategy is to consider pluri-anticanonical sections with polynomial growth.
The key new result is a uniform multiplicity estimate for pluri-anicanonical sections (proposition \ref{p29}). 
This provides the dimension estimate for polynomial growth pluri-anticanonical sections, without the extra assumption $\int_MRic^n<+\infty$ (compare theorem $2.2$ of \cite{[Mo]}).

\medskip

This paper is organized as follows: In section $2$, we prove results which are essential to all three parts of theorem \ref{thm1}.
The proof of part I is presented in section $3$. The main point is to prove a properness result, which is a modification of theorem $6.1$ of \cite{[L3]}.
In section $4$, we prove part II of theorem \ref{thm1} and corollary \ref{cor1}. We shall follow part I, \cite{[DS2]} and \cite{[CSW]}.
In the last section, we prove part III of theorem \ref{thm1}. We follow Mok's argument \cite{[Mo]}, with indication of the differences.

Here are some conventions in this paper.
Let $e_\alpha$ be a local unitary frame of $T^{1, 0}M$ and $s$ be a smooth tensor on $M$.
Define $\Delta s = s_{\alpha\overline\alpha}+s_{\overline\alpha\alpha}$. Note this is twice the Laplacian defined in \cite{[NT1]}. Also define $|\nabla u|^2 = 2u_\alpha u_{\overline{\beta}}g^{\alpha\overline{\beta}}$. Let $\dashint$ be the average. We will denote by $\Phi(u_1,..., u_k|....)$ any nonnegative functions depending on $u_1,..., u_k$ and some additional parameters such that when these parameters are fixed, $$\lim\limits_{u_k\to 0}\cdot\cdot\cdot\lim\limits_{u_1\to 0}\Phi(u_1,..., u_k|...) = 0.$$  Let $C(\cdot, \cdot, .., \cdot)$ and $c(\cdot, \cdot, .., \cdot)$ be large and small positive constants respectively, depending only on the parameters. The values might change from line to line.

\medskip

\begin{center}
\bf  {\quad Acknowledgments}
\end{center}
The author would like to thank Professors H. Hein, A. Naber, J. Lott and S. Zelditch for their interests in this work.

\bigskip

\section{\bf{Preliminary results}}
In this section, we prove several results which are essential for the proof of theorem \ref{thm1}.

First we need the following version of three circle theorem: 
\begin{prop}\label{p1}
Let $M^n$ be a complete noncompact K\"ahler manifold, $p\in M$, $r(x)=dist(p, x)$. Suppose there exist constants $\epsilon, A >0$ such that for any $e_i\in T_x^{1,0}M$ with unit length, $R_{i\overline{i}i\overline{i}} \geq -\frac{A}{(r+1)^{2+\epsilon}}$. Then for any holomorphic function $f$ on $M$, $\log M(r)$ is convex in terms of $h(r)$, where $M(r) = \sup\limits_{B(p, r)} |f|$ and \begin{equation}\label{1}\begin{aligned}
h(r) = \int\limits^r_{1}\frac{e^{\frac{2A}{\epsilon(t+1)^{\epsilon}}}}{t}dt\end{aligned}\end{equation} for $r\geq 1$. In particular, $dim(\mathcal{O}_d(M))\leq Cd^n$, where  $\mathcal{O}_d(M)= \{f$ holomorphic on $M||f(x)|\leq C(1+r(x))^{d+\epsilon}\}$ for any $\epsilon>0$.\end{prop}
\begin{proof}
This is a combination of theorem $8$ and theorem $11$ in \cite{[L1]}. Notice that we have multiplied the original $h(r)$ in (8) of \cite{[L1]} by $e^{\frac{2A}{\epsilon}}$.
\end{proof}
It is clear that there exists a constant $b$ so that \begin{equation}\label{2}\lim\limits_{r\to\infty}(h(r) - \log r - b) = 0.\end{equation} Given $d>0, R>2$, define $d_R = \frac{d\log 2}{h(R)-h(\frac{R}{2})}$.
Then $d_R\to d$ as $R\to\infty$.
We have the following
\begin{cor}\label{cor21} Under the assumption of proposition \ref{p1}, 
let $f$ be a holomorphic function on $B(p, R)$ so that $\frac{M(R)}{M(\frac{R}{2})}\leq 2^d$. Assume $M(\rho) = 1$ for some $1<\rho<\frac{R}{10}$. Then for any $2\rho\leq r\leq R$, $M(r)\leq r^{d_R}$.
\end{cor}
\begin{proof}
From the assumption and proposition \ref{p1}, we see that $\log M(r) - d_Rh(r)$ is non-increasing for $\rho\leq r\leq \frac{R}{2}$. In particular, 
\begin{equation}\label{3}\log M(r) - d_Rh(r) \leq \log M(\rho) - d_Rh(\rho)\leq 0.\end{equation} Thus by (\ref{2}),  \begin{equation}\label{4}\log M(r) - d_R\log r \leq 0.\end{equation} This completes the proof of the corollary.
\end{proof}

The following proposition could be found in \cite{[Mo1]} or \cite{[L6]}. For reader's convenience, we include the proof.
\begin{prop}\label{p2}[Mok]
Let $f, g$ be polynomial growth holomorphic functions on a complete K\"ahler manifold $M$ with $Ric\geq 0$. Suppose $h=\frac{f}{g}$ is holomorphic, then $h$ is of polynomial growth.
\end{prop}
\begin{proof}
Let us say $f(p), g(p)\neq 0$.
Set $F_1(x) = \log |f(x)|^2+\int_{B(p, R)}G_R(x, y)\Delta\log |f(y)|^2, F_2(x) = \log |g(x)|^2+\int_{B(p, R)}G_R(x, y)\Delta\log |g(y)|^2$.
\begin{claim}
For large $R$ and $i = 1, 2$, on $B(p, \frac{R}{2})$, $-C\log R\leq F_i(x)\leq C\log R$.
\end{claim}
\begin{proof}
Note that $F_i(x)$ is harmonic on $B(p, R)$. Now maximum principle says that $F_i(x)\leq C\log R$ on $B(p, R)$. Let $H_i  = C\log R-F_i\geq 0$. Then gradient estimate implies that on $B(p, \frac{3}{4}R)$, $|\nabla \log H_i|\leq \frac{C_1}{R}$. Observe $H_i(p)\leq C\log R$.  Then the harnack inequality implies that $H_i\leq C_2\log R$ on $B(p, \frac{R}{2})$. This completes the proof of the claim.
\end{proof}
It is clear that on $B(p, \frac{R}{2})$ $\log |h(x)|^2 \leq F_1(x)-F_2(x)\leq C\log R$ ($C$ is independent of $R$). The proof of the proposition is complete.\end{proof}

\medskip

In the rest of this section, unless otherwise stated, we assume
 $(M^n, p)$ is a complete noncompact K\"ahler manifold with nonnegative Ricci curvature, maximal volume growth and the bisectional curvature $BK\geq -\frac{C}{r^2}$ where $r(x) = dist(x, p)$. Given any sequence $\mu_i\to\infty$, set $(M_i, p_i, d_i) = (M, p, \frac{d}{\mu_i})$. Set $r_i$ be the distance to $p_i$. By passing to subsequence, we assume $(M_i, p_i, d_i)\to (M_\infty, p_\infty, d_\infty)$ in the pointed Gromov-Hausdorff sense. Cheeger-Colding theorem \cite{[CC1]} says $(M_\infty, p_\infty, d_\infty)$ is a metric cone.  We have the following
\begin{prop}\label{p3}
Given any $R>100$, for sufficiently large $i$, we can find a plurisubharmonic function $u$ on $B(p_i, R)$ with \begin{equation}\label{5}|u-r_i^2| \leq  \Phi(\frac{1}{i})R^2,\end{equation} \begin{equation}\label{6}|\nabla u|^2 - 4r_i^2\leq \Phi(\frac{1}{i})R^2.\end{equation} Furthermore, if $\omega_i$ is the K\"ahler form on $M_i$, then on $B(p_i, R)\backslash B(p_i, \Phi(\frac{1}{i})R)$, \begin{equation}\label{7}\sqrt{-1}\partial\overline\partial u\geq (1-\Phi(\frac{1}{i}))\omega_i.\end{equation}\end{prop}
\begin{remark}
In \cite{[L3]} and \cite{[L4]}, we applied the parabolic method of Ni-Tam \cite{[NT1]} to obtain such function. As we shall see in the proof below, elliptic method suffices.
\end{remark}
\begin{cor}\label{cor22}
There exists some $R>0$ so that any compact subvariety of positive dimension in $M$ is contained in $B(p, R)$.
\end{cor}
\begin{proof}
Assume that there exists a sequence of compact subvarieties $V_i$ of positive dimension so that $\mu_i = d(p, V_i)\to\infty$. 
By the scaling as in the proposition above, we may assume that $V_i$ is a subvariety of $M_i$ and $d(p_i, V_i) = 1$. Then for sufficiently large $i$, there exists a psh function $u$ on $B(p_i, 200)$  satisfying 
\begin{equation}\label{8}
|u-r_i^2| \leq  \Phi(\frac{1}{i}).
\end{equation}
Furthermore, on $B(p_i, 200)\backslash B(p_i, \Phi(\frac{1}{i}))$,
\begin{equation}\label{9}
\sqrt{-1}\partial\overline\partial u\geq (1-\Phi(\frac{1}{i}))\omega_i.
\end{equation}
Then the maximum of $u$ on $V_i$ is achieved somewhere on $B(p_i, 200)\backslash B(p_i, 0.2)$, where $u$ is strictly plurisubharmonic.
This is a contradiction.

\end{proof}

Now we prove proposition \ref{p3}.
\begin{proof}
Following \cite{[CC1]}, we solve the equation \begin{equation}\label{10}\Delta \hat{u} = 2n\end{equation} on $B(p_i, 2R)\backslash B(p_i, 1)$ with boundary value
$\hat{u}= 4R^2$ on $\partial B(p_i, 2R)$, $\hat{u}= 1$ on $\partial B(p_i, 1)$.
Laplace comparison implies that \begin{equation}\label{11}\hat{u}\leq r_i^2.\end{equation} Let $G_i(x, y)$ be the Green function on $B(p_i, 2R)\backslash B(p_i, 1)$.
Then \begin{equation}\label{12}v(x) = \hat{u}(x) +\int_{B(p_i, 2R)\backslash B(p_i, 1)}G_i(x, y)\Delta \hat{u}(y)dy\end{equation} is harmonic. By checking the boundary values, we find \begin{equation}\label{13}|v|\leq 4R^2\end{equation} on $B(p_i, 2R)\backslash B(p_i, 1)$. Now from the maximum principle, for $x, y\in B(p_i, 2R)\backslash B(p_i, 1)$, \begin{equation}\label{14}G_i(x, y)\leq \hat{G}_i(x, y),\end{equation} where $\hat{G}_i(x, y)$ is the Green function on $B(p_i, 3R)$. 
Since the manifold has maximal volume growth, we find \begin{equation}\label{15}\int_{B(p_i, 2R)}\hat{G}_i(x, y)dy\leq C\end{equation} where $C$ is independent of $i$ and $x\in B(p_i, 2R)$. 
Then from (\ref{12}), we find that \begin{equation}\label{16}|\hat{u}|\leq C(R).\end{equation}
Following proposition $4.35$ of \cite{[CC1]}, we have \begin{equation}\label{17}\int_{B(p_i, 2R)\backslash B(p_i, 1)} |\nabla \hat{u}-\nabla r_i^2|^2<\Phi(\frac{1}{i}).\end{equation} Then similar as proposition $4.50$ of \cite{[CC1]}, on $B(p_i, 1.75R)\backslash B(p_i, 1.25)$, \begin{equation}\label{18}|\hat{u}-r_i^2|<\Phi(\frac{1}{i}).\end{equation}  The argument uses Poincare inequality and the gradient estimate of Cheng-Yau. As in (4.56) of \cite{[CC1]}, the Bochner formula for $\Delta|\nabla \hat{u}|^2$ gives that \begin{equation}\label{19}\int_{B(p_i, 1.5R)\backslash B(p_i, 2)} |\nabla^2 \hat{u}-g_i|^2<\Phi(\frac{1}{i}).\end{equation} This of course implies that 
\begin{equation}\label{20}\int_{B(p_i, 1.5R)\backslash B(p_i, 2)} |\hat{u}_{\alpha\overline\beta}-g_{i, \alpha\overline\beta}|^2<\Phi(\frac{1}{i}).\end{equation}

\begin{claim}\label{cl1}On $B(p_i, 1.5R)\backslash B(p_i, 2)$, 
$\Delta |\hat{u}_{\alpha\overline\beta}-g_{i, \alpha\overline\beta}|\geq -C|\hat{u}_{\alpha\overline\beta}-g_{i, \alpha\overline\beta}|$,
 $|\nabla \hat{u}|\leq 2r_i+\Phi(\frac{1}{i})$.
\end{claim}
\begin{remark}
As we will see,  the bisectional curvature lower bound can be replaced by the quadratic orthogonal bisectional curvature lower bound.\end{remark}
\begin{proof}
Set \begin{equation}\label{221}\hat{v}_{\alpha\overline\beta} = \hat{u}_{\alpha\overline\beta}-g_{i, \alpha\overline\beta}.\end{equation} Let us diagonalize $\hat{v}$ at a point so that \begin{equation}\label{222}\hat{v}_{\alpha\overline\beta} = a_{\alpha}\delta_{\alpha\beta}.\end{equation} As $BK(M)\geq -\frac{C}{r^2}$, on $B(p_i, 2R)\backslash B(p_i, 1)$, the bisectional curvature has a lower bound $-C$.
The calculation of \cite{[MSY]} on page $186$ and $187$  ((iv) and (vi)) gives 
\begin{equation}\label{223}\frac{1}{2}\Delta ||\hat{v}||^2 = |\hat{v}_{\alpha\overline\beta\gamma}|^2+|\hat{v}_{\alpha\overline\beta\overline\gamma}|^2 + 2R_{i, \alpha\overline\alpha\beta\overline\beta}(a_\alpha-a_\beta)^2\geq |\hat{v}_{\alpha\overline\beta\gamma}|^2+|\hat{v}_{\alpha\overline\beta\overline\gamma}|^2-C||\hat{v}||^2.\end{equation}

Note \begin{equation}\label{224}|\nabla ||\hat{v}|||^2\leq |\hat{v}_{\alpha\overline\beta\gamma}|^2+|\hat{v}_{\alpha\overline\beta\overline\gamma}|^2.\end{equation}
The first statement follows from (\ref{223}) and (\ref{224}).
For the second statement, we use an argument similar to Cheeger-Naber \cite{[CN]}, page 1116-1117.
By the Bochner formula, \begin{equation}\label{225}\Delta(|\nabla \hat{u}|^2-4\hat{u})\geq 0.\end{equation} Notice (\ref{17}) and (\ref{18}) imply \begin{equation}\label{226}\int_{B(p_i, 1.6R)\backslash B(p_i, 1.5)} ||\nabla \hat{u}|^2-4\hat{u}||^2<\Phi(\frac{1}{i}).\end{equation} Therefore, by applying the mean value inequality to 
$\max(|\nabla \hat{u}|^2 - 4\hat{u}, 0)$, we finish the proof of the second statement.

\end{proof}
Now by applying the parabolic mean value inequality to (\ref{20}) and claim \ref{cl1}, we find that 
\begin{equation}\label{227}\hat{u}_{\alpha\overline\beta}\geq (1-\Phi(\frac{1}{i}))g_{i, \alpha\overline\beta}\end{equation} on $B(p_i, 1.2R)\backslash B(p_i, 3)$. Set $u = \max(\hat{u} - 5, 0)$. Then for large $i$, $u$ is a psh function on $B(p, R)$ with \begin{equation}\label{228}
\sqrt{-1}\partial\overline\partial u\geq (1-\Phi(\frac{1}{i}))\omega_i
\end{equation}
on $B(p_i, R)\backslash B(p_i, 10)$. Moreover,
\begin{equation}\label{229}
|u-r_i^2|\leq 5+\Phi(\frac{1}{i}), |\nabla u|^2- 4r_i^2\leq \Phi(\frac{1}{i}).
\end{equation}
Now we can solve new equations $\Delta \hat{u} = 2n$ on $B(p_i, 2R)\backslash B(p_i, \Phi(\frac{1}{i}))$ where inner radius $\Phi(\frac{1}{i})$ goes to zero sufficiently slow. Then we can apply the same argument as before. Finally, if we set $u = \max(\hat{u}-\Phi(\frac{1}{i}), 0)$ for some other $\Phi(\frac{1}{i})$ going to zero sufficiently slow, then $u$ satisfies (\ref{5}), (\ref{6}) and (\ref{7}).
This completes the proof of proposition \ref{p3}.

\end{proof}
Now let us adopt the assumption in proposition \ref{p3}.  Recall $(M_i, p_i, d_i) = (M, p, \frac{d}{\mu_i})\to (M_\infty, p_\infty, d_\infty)$. Pick any points $q_1 \neq q_2\in \partial B(p_\infty, 5)$. Let $q^i_j\in M_i$ be so close to $q_j$. According to proposition $6.1$ in \cite{[L3]} (see also proposition $5.2$ in \cite{[L4]}), there exists $\delta>0$ independent of $i$ so that for all sufficiently large $i$, we can find a smooth function $v^i_j$ with isolated singularty at $q^i_j$ and supported on $B(q, \delta)$.
Moreover, $e^{-v^i_j}$ is not locally integrable at $q^i_j$ and $\sqrt{-1}\partial\overline\partial v^i_j\geq -C\omega_i$. Let $u$ be the psh function constructed in proposition \ref{p3}. Thus there exists $C>0$ independent of $i$ so that 
$\sqrt{-1}\partial\overline\partial (Cu+v^i_1+v^i_2)\geq \omega_i$ on $B(p_i, R)\backslash B(p_i, 1)$ and $(Cu+v^i_1+v^i_2)$ is psh on $B(p_i, R)$.
By using the same argument as in proposition $5.2$ of \cite{[L4]}, we can holomorphically separate points on $B(p_i, 2)\backslash B(p_i, 1)$. Also we can separate $p_i$ from $B(p_i, 2)\backslash B(p_i, 1)$. We have the following
\begin{prop}\label{p4}
There exist
$N\in\mathbb{N}, A>5$ depending only on $M$ so that for all sufficiently large $i$, there exist holomorphic functions $g^1_i, .., g^N_i$ on $B(p_i, 6A)$ and
\begin{equation}\label{230}
g^j_i(p_i)=0,  \sup\limits_{B(p_i, 1)}\sum\limits_{j=1}^N|g_i^j|^2 = 1.\end{equation}
  \begin{equation}\label{231}\min\limits_{x\in\partial B(p_i, 2A)}\sum\limits_{j=1}^N|g_i^j(x)|^2> 4\sup\limits_{x\in B(p_i, 1)}\sum\limits_{j=1}^N|g_i^j(x)|^2=4.\end{equation}  Furthermore, for all $j$, \begin{equation}\label{232}\frac{\sup\limits_{x\in B(p_i, 3A)}|g_i^j(x)|^2}{\sup\limits_{x\in B(p_i, 2A)}|g_i^j(x)|^2}\leq C=C(M).\end{equation} 
\end{prop}
\begin{proof}
Given proposition \ref{p3}, the argument is almost the same as in proposition $6.1$ of \cite{[L3]}. One point is that the three circle theorem in $(6.32)$ of \cite{[L3]} should be replaced by the following 
\begin{lemma}\label{l21}Let $R_0>2$ be a constant.
For sufficiently large $i$, if $h$ is a holomorphic function on $B(p_i, R_0)$ satisfying $\frac{\sup\limits_{B(p_i, R_0)}|h|}{\sup\limits_{B(p_i, \frac{R_0}{2})}|h|}\leq C$ and $\sup\limits_{B(p_i, 1)}|h| = 1$, then $\sup\limits_{B(p_i, R_0)}|h|\leq C(R_0)$.
\end{lemma}
\begin{proof}
Let $i\to\infty$. If we first normalize $h_i = c_ih$ so that $\sup\limits_{B(p_i, R_0)} |h_i| = 1$, then by passing to subsequence, $h_i$ converges uniformly on each compact set to a harmonic function $h_\infty$ on $B(p_\infty, R_0)$. Then $\sup\limits_{B(p_\infty, \frac{R_0}{2})}|h_\infty|\geq \frac{1}{C}$ and \begin{equation}\label{233} \frac{\int_{B(p_\infty, 0.8R_0)}|h_\infty|^2}{\int_{B(p_\infty, 0.6R_0)}|h_\infty|^2}\leq C.\end{equation}
By three circle theorem of harmonic functions over metric cones, we obtain that $\sup\limits_{B(p_\infty, 1)}|h_\infty|\geq c>0$. This complete the proof of the lemma. \end{proof}
This suffices to prove proposition \ref{p4}.
\end{proof}
The above argument shows that $|g^j_i|\leq C(M)$ on $B(p_i, 3A)$. Set $G_i = (g^1_i, .., g^N_i)$. Let $B(p_i, 1)\subset\Omega'_i\subset\subset B(p_i, 2A)$ be the connected component of $G_i^{-1}(B_{\mathbb{C}^N}(0, 1))$ containing $p_i$.  Then $G_i: \Omega'_i\to B_{\mathbb{C}^N}(0, 1)$ is proper. Let $V_i = G_i(\Omega'_i)\subset B_{\mathbb{C}^N}(0, 1)$ be the subvariety. Note $V_i$ has dimension $n$. To see this, take $x\in \partial B_{\mathbb{C}^N}(0, \frac{1}{2})\cap V_i$. Then $G_i^{-1}(x)\cap \Omega'_i$ is a compact subvariety of $\Omega'_i$, in particular, a compact subvariety $W$ of $M$. As $i$ is large, each irreducible component of $W$ must be far away from $p$. According to corollary \ref{cor22}, $W$ is zero dimensional. Hence, $V_i$ has dimension $n$.

According to local parametrization theorem $4.19$ in \cite{[D]}, given any analytic variety $0\in V^n\subset B_{\mathbb{C}^N}(0, 1)$, we can locally properly project $(V^n, 0)$ to a complex $n$-dimensional plane through $0$.
Since the volume of $V_i$ is uniformly bounded from above, Bishop's theorem \cite{[B]} says $V_i$ form a relatively compact set of the subvarieties of dimension $n$, through $0$ in unit ball of $\mathbb{C}^N$.  Then we can find $c_1>0, c_2>0$, a compact set $K$ of $GL(N, \mathbb{C})$ depending only on $M$ so that the following hold for all large $i$:

\medskip

1.  $B^i\in K$

2.  Set $\pi_i': (g^1_i, .., g^N_i)\to (h^1_i, ..., h^n_i)$ be given by $h^k_i = \sum\limits_{j=1}^NB^i_{jk}g^j_i$.

3. For $(g^1_i, ..., g^N_i)\in V_i$, if $\sum\limits_{j=1}^N|g^j_i|^2\leq c_1<1$, $\sum\limits_{k=1}^n|h^k_i|^2\geq c_2\sum\limits_{j=1}^N|g^j_i|^2$.

\medskip

\begin{claim}\label{cl0}
Set $\pi_i = (h^1_i, .., h^n_i): B(p_i, 3A)\to \mathbb{C}^n$. Set $r_0 = \frac{1}{12}\sqrt{c_1c_2}$. Let $\Omega_i$ be the connected component of $\pi_i^{-1}(B_{\mathbb{C}^n}(0, 6r_0))$ containing $p_i$. Then for all sufficiently large $i$, $\Omega_i \subset B(p_i, 2A)$.
\end{claim}
\begin{proof}
It is clear that $\pi_i = \pi_i'\circ G_i$. Assume the claim is false. Then there exists a curve $\gamma_i\subset \Omega_i$ connecting $p_i$ and some point $q_i$ on $\partial B(p_i, 2A)$. Observe that by (\ref{231}), $|G_i(q_i)|\geq 2$. As $|G_i(p_i)| = 0$, we can find a point $t_i\in\gamma_i$ so that $|G_i(t_i)|^2 = c_1$. Then according to item $3$ above, $|\pi_i(t_i)|\geq \sqrt{c_1c_2}>6r_0$. This is a contradiction.
\end{proof}

Since the matrix $B^i$ and $g^j_i$ are bounded, there exists a constant $r_1$ depending only on $M$ so that $\pi_i(B(p_i, 2r_1))\subset B_{\mathbb{C}^n}(0, r_0)$.
Let us summarize the results as 
\begin{prop}\label{p0}
There exist positive constants $r_0, r_1$ depending only on $M$ so that for all large $i$, $\pi_i(B(p_i, 2r_1))\subset B_{\mathbb{C}^n}(0, r_0)$, $\Omega_i\subset B(p_i, 2A)$.\end{prop}

Proposition \ref{p0} will be used later.
\begin{prop}\label{p5}
Any tangent cone of $(M, p)$ at infinity is homeomorphic to a normal affine algebraic variety.
\end{prop}
\begin{proof}
Let $(M_i, p_i, d_i) = (M, p, \frac{d}{\mu_i})$ for some sequence $\mu_i\to\infty$. Assume $(M_i, p_i, d_i)$ converges in the pointed Gromov-Hausdorff limit $(M_\infty, p_\infty, d_\infty)$. By proposition \ref{p4}, we can find $g^i_1, .., g^i_N$ satisfying (\ref{230}), (\ref{231}) and (\ref{232}).
Then we can apply the same argument as in \cite{[L4]} to prove the proposition.

\end{proof}

\begin{prop}\label{p6}
$r^{2-2n}\int_{B(p, r)}S\leq C(M)$ for any $r>0$, where $S$ is the scalar curvature.
\end{prop}
\begin{proof}
Let $(M_i, p_i, d_i) = (M, p, \frac{d}{\mu_i})$ for some sequence $\mu_i\to\infty$. Then for large $i$, $B(p_i, 1)$ is sufficiently close to a metric cone. Hence, we can apply proposition \ref{p3}. By Gromov compactness theorem and Cheeger-Colding theory, we can find a point $q\in \partial B(p_i, \frac{1}{2})$ so that $B(q, \delta)$ is $\epsilon_0\delta$-Gromov-Hausdorff close to $B_{\mathbb{C}^n}(0, \delta)$, where $\delta>0$ is independent of $i$ and $\epsilon_0$ is a very small number depending only on $n$. According to \cite{[L4]}, there exists a holomorphic chart $(z_1, ..., z_n)$ on $B(q, \epsilon_1\delta)$, where $\epsilon_1$ depends only on $n$. Furthermore, 
\begin{equation}\label{234}\dashint_{B(q, \epsilon_1\delta)}|\langle dz_k, \overline {dz_l}\rangle-\delta_{kl}|^2 <\Phi(\epsilon_1).\end{equation}
Then for sufficiently small $\epsilon_1$, there exists a point $q'\in B(q, \frac{\epsilon_1\delta}{2})$ so that $|dz_1\wedge\cdot\cdot\wedge dz_n(q')|\geq \frac{1}{2}$. Set $\Omega = dz_1\wedge\cdot\cdot\wedge dz_n$.
We can solve $\overline\partial s' = \overline\partial (v^i\Omega)$, where $v^i$ is analogous to the cut-off function $v^i_j$ above proposition \ref{p4}.
Then we find $s\in\Gamma(B(p_i, 1), K)$ ($K$ is the canonical line bundle) so that \begin{equation}\label{235}|s(q')|\geq \frac{1}{2}\end{equation} and  on $B(p_i, 0.9)$, \begin{equation}\label{236}\sup |s|\leq C,\end{equation} where $C$ is independent of $i$. The Poincare-Lelong equation says \begin{equation}\label{237}\frac{\sqrt{-1}}{2\pi}\partial\overline\partial\log|s|^2 = [D]+Ric(M_i),\end{equation} where $D$ is the divisor of $s$. Let $G(x, y)$ be the Green function on $B(p_i, 0.8)$. Let $S_i$ be the scalar curvature on $M_i$. By taking the trace and integrate, we find 
\begin{equation}\label{238}\int_{B(p_i, 0.8)} G(q', y)\Delta\log|s(y)|^2dy \geq 2\pi\int_{B(p_i, 0.8)} G(q', y)S_i(y)dy\geq c(M)\int_{B(p_i, \frac{1}{2})}S_i(y)dy.\end{equation} Define 
\begin{equation}\label{239}F(x) = \log |s(x)|^2+\int_{B(p_i, 0.8)} G(x, y)\Delta\log |s(y)|^2dy.\end{equation} Then $F(x)$ is harmonic. Maximum principle says $F$ bounded from above by $C$ on $B(p_i, 0.8)$. Therefore \begin{equation}\label{240}\int G(q', y)\Delta\log |s(y)|^2dy + \log |s(q')|^2 \leq C.\end{equation} As $|s(q')|\geq \frac{1}{2}$, we find $\int_{B(p_i, \frac{1}{2})}S_i \leq C$. The proposition is proved.

\end{proof}

\begin{cor}\label{cor23}
Let $(X^n, p)$ be a complete noncompact K\"ahler manifold with nonnegative bisectional curvature and maximal volume growth. Then there exists some $C=C(X)>0$ so that for any $q\in X$, any $r>0$, $r^{2-2n}\int_{B(q, r)}S\leq C$.
\end{cor}
\begin{remark}
This result was proved in \cite{[N]} when $X$ has bounded curvature. In \cite{[L2]}, the result was proved for $C$ depending on $q$. 
\end{remark}
\begin{proof}
Let $v = \lim\limits_{r\to\infty}\frac{vol(B(p, r))}{r^{2n}}>0$. By volume comparison, we can find $N\in\mathbb{N}$ so that for any $q\in M$, $r>0$, there exists $1\leq l\leq N$ and that $B(q, 2^lr)$ is $\epsilon r$-Gromov-Hausdorff close to a metric cone. Here $\epsilon= \epsilon(n, v)$ is so small that the argument in proposition \ref{p6} can be applied. Then we find that \begin{equation}\label{241}(2^{l-1}r)^{2-2n}\int_{B(q, 2^{l-1}r)}S\leq C.\end{equation} Note by Gromov compactness theorem, such $C$ depends only on $n, v$.  This concludes the proof of the corollary.
\end{proof}

Recall theorem $1.1$ by Ni-Shi-Tam \cite{[NST]}:
\begin{prop}\label{p7}
Let $(X^n, p)$ be a complete noncompact K\"ahler manifold with nonnegative Ricci curvature and maximal volume growth. Let $f$ be a smooth nonnegative function on $X$. 
Set $k(x, r) = \dashint_{B(x, r)}f$ and $k(r) = \dashint_{B(p, r)}f$. Suppose $\int_0^\infty k(t)dt<+\infty$, then there exists a solution $u$ to $\Delta u = f$ so that \begin{equation}-C(r\int_{2r}^\infty k(t)dt+\int_0^{\frac{1}{2}r}tk(x, t)dt)+c\int_0^{2r}tk(t)dt\leq u(x)\leq C(r\int_r^\infty k(t)dt + \int_0^{2r}tk(t)dt),\end{equation} where $C, c$ are positive constant independent of $x$.
\end{prop}

We set $f$ to be the scalar curvature on $M$. Proposition \ref{p6} says $k(t)\leq \frac{C}{1+t^2}$. Then by proposition \ref{p7}, we obtain
\begin{prop}\label{p28}
There exists a function $\rho'$ so that $\Delta \rho' = \pi S$ and $\rho'\leq C\log(r+2)$. If in addition, $S\leq\frac{C}{r^2}$, then $C\log(r+2)\geq \rho'\geq c\log(r+1)-C$.
\end{prop}

The following proposition plays an important role in part II and III. It is an improvement of theorem $2.3$ in \cite{[Mo]}.
\begin{prop}\label{p29} Let $\rho'$ be defined as in proposition \ref{p28}.
Let $f$ be a holomorphic section in $K^{-q}(M)$ and $V$ be the divisor of $f$. Assume that $|fe^{q\rho'}|\leq C(1+r)^d$ on $M$. Then for $x$ not on compact subvarieties of positive dimension, $Mul_x(V)\leq Cd$, where $C$ depends only on $M$, $Mul_x(V)$ is the multiplicity of $V$ at $x$.
\end{prop}
\begin{proof}
  Define \begin{equation}\label{243}M'(r) = \sup\limits_{B(p, r)} |f|e^{q\rho'}.\end{equation}
Let $A$ be the constant in proposition \ref{p4}, $r_0, r_1$ be the constants in proposition \ref{p0}. Both constants depend only on $M$.
 Since $|fe^{q\rho'}|\leq C(1+r)^d$, we can find a sequence $\mu_i\to\infty$ with \begin{equation}\label{244}\frac{M'(20A\mu_i)}{M'(r_1\mu_i)}\leq (\frac{20A}{r_1})^{2d}.\end{equation}  Set $(M_i, p_i, d_i) = (M, p, \frac{d}{\mu_i})$. Let $g^j_i$ be as in proposition \ref{p4}. Let $\pi_i$, $\Omega_i$ be defined as in claim \ref{cl0}. From now on, we shall restrict $\pi_i$ to $\Omega_i$. Then $\pi_i: \Omega_i\to B_{\mathbb{C}^n}(0, 6r_0)$ is a proper map.
 
 According to proposition \ref{p0}, for large $i$, \begin{equation}\label{245}\pi_i^{-1}(B_{\mathbb{C}^n}(0, 5r_0)\backslash B_{\mathbb{C}^n}(0, 1.4r_0))\subset B(p_i, 2A)\backslash B(p_i, 2r_1).\end{equation} Recall $V$ is the divisor of $f$. let $V'$ be the union of irreducible components of $V$ containing $x$. Then $Mul_x(V')=Mul_x(V)$.
As $x$ is a fixed point on $M$, $x\in B(p_i, \frac{r_1}{10})$ for $i$ large. Let $x_i = \pi_i(x)\to 0\in\mathbb{C}^n$.
Since $x$ is not on any positive dimensional compact subvarieties and $\pi_i$ is proper, each irreducible component of $\pi_i(V')$ has dimension $n-1$. The multiplicity of $\pi_i(V')$ at $x_i$ will be at least $m=Mul_x(V')=Mul_x(V)$ ( just apply the line test, consider a line intersecting $\pi_i(V)$ at $x_i$, then pull back). The standard Lelong monotonicity implies that \begin{equation}\label{246}vol(\pi_i(V')\cap B(x_i, 4r_0)\backslash B(x_i, 2r_0))\geq c(r_0)m>0.\end{equation} Recall $x_i\to 0\in\mathbb{C}^n$. Then (\ref{245}) implies that \begin{equation}\label{247}\pi_i^{-1}(B(x_i, 4r_0)\backslash B(x_i, 2r_0))\subset B(p_i, 2A)\backslash B(p_i, 2r_1).\end{equation}
Recall $|g^j_i|\leq C(M)$ on $B(p_i, 3A)$. By the gradient estimate for $g^j_i$, \begin{equation}\label{248}vol(V'\cap B(p_i, 2A)\backslash B(p_i ,2r_1))\geq c(M, A, r_0)m>0.\end{equation} Therefore, \begin{equation}\label{249}vol(V'\cap B(p, 2A\mu_i)\backslash B(p, 2\mu_ir_1))\geq c(M, A, r_0, r_1)m\mu_i^{2n-2}.\end{equation}  
Poincare-Lelong equation says \begin{equation}\label{250}\frac{\sqrt{-1}}{2\pi}\partial\overline\partial\log |f|^2 = -qRic + [V]. \end{equation} Let $G_{20A\mu_i}(z, y)$ be the Green function on $B(p, 20A\mu_i)$. Define $F(y) = |f(y)|e^{q\rho'(y)}$.
As $\Delta \rho' = \pi S$, \begin{equation}\label{251}\frac{1}{\pi}\int_{B(p, 2A\mu_i)\backslash B(p, 2r_1\mu_i)}G_{20A\mu_i}(z, y)\Delta\log F(y)dy\geq  \int_{V\cap B(p, 2A\mu_i)\backslash B(p, 2r_1\mu_i)}G_{20A\mu_i}(z, y)dy.\end{equation}
Now we apply similar argument as in proposition \ref{p2}. Let $H_i(z)$ be the harmonic function defined by
\begin{equation}\label{252}H_i(z) = \log F(z) + \int_{B(p, 20A\mu_i)}G_{20A\mu_i}(z, y)\Delta\log F(y)dy.\end{equation} According to maximum principle and (\ref{244}), on $B(p, 20A\mu_i)$, \begin{equation}\label{253} H_i\leq dC(\frac{20A}{r_1})+\log M'(r_1\mu_i).\end{equation}
Let $z^i_0\in\partial B(p, r_1\mu_i)$ be so that $F(z^i_0) = M'(r_1\mu_i)$. Set $z=z_0^i$ in (\ref{252}). Then (\ref{251}), (\ref{252}) and (\ref{253}) imply that 
\begin{equation}\label{254}\begin{aligned}
&\int_{V\cap B(p, 2A\mu_i)\backslash B(p, 2r_1\mu_i)}G_{20A\mu_i}(z_0^i, y)dy\\&\leq \frac{1}{\pi}\int_{B(p, 20A\mu_i)}G_{20A\mu_i}(z_0^i, y)\Delta\log F(y)dy\\&\leq dC(\frac{20A}{r_1}).\end{aligned}\end{equation}
Notice for $y\in B(p, 2A\mu_i)\backslash B(p, 2r_1\mu_i)$, $G_{20A\mu_i}(z_0^i, y)\geq \frac{c(M, A, r_1)}{\mu_i^{2n-2}}$. Then we find that
\begin{equation}\label{255}vol(V\cap B(p, 2A\mu_i)\backslash B(p, 2r_1\mu_i))\leq C(M, A, r_1)d\mu_i^{2n-2}.\end{equation}  The proposition follows from (\ref{249}) and (\ref{255}).

\end{proof}
\begin{definition}\label{d1}
Let $\mathcal{H}^q_d(M) = \{f\in K^{-q}(M)||f(x)e^{q\rho'(x)}|\leq C(1+r(x))^{d+\epsilon}\}$ for any $\epsilon>0$.
Let $\mathcal{O}_d(M)= \{f$ holomorphic on $M||f(x)|\leq C(1+r(x))^{d+\epsilon}\}$ for any $\epsilon>0$. For a polynomial growth holomorphic function $f$, let $deg(f) = \inf\limits_{f\in\mathcal{O}_d(M)} d$. \end{definition}

\begin{cor}\label{cor24}
$dim(\mathcal{H}^q_d(M))\leq Cd^n$, where $C$ is independent of $q$ and $d$. In particular, if we set $q=0$, then $dim(\mathcal{O}_d(M))\leq Cd^n$.\end{cor}

\section{\bf{Proof of theorem \ref{thm1}, part I}}
The crucial proposition is the following:
\begin{prop}\label{p31}
Under the assumption of part I, we can find polynomial growth holomorphic functions $g_1, ..., g_N$ so that $(g_1, ..., g_N)$ is a a proper map to $\mathbb{C}^N$.
\end{prop}
\begin{proof}
The argument is almost the same as theorem $6.1$ of \cite{[L3]}. Let us sketch the argument, for completeness.  For any sequence $\mu_i\to\infty$
let $(M_i, p_i, d_i) = (M, p, \frac{d}{\mu_i})$. By passing to subsequence, we may assume $(M_i, p_i, d_i) \to (M_\infty, p_\infty, d_\infty)$ in pointed Gromov-Hausdorff sense. By proposition \ref{p4}, there exist $K, 0<a<\frac{1}{10}$ independent of $i$ and holomorphic function $\tilde{f}^j_i$ $(1\leq j\leq K)$ on $B(p_i, 6)$ so that \begin{equation}\label{31}
\tilde{f}^j_i(p_i)=0,  \max_j\sup\limits_{B(p_i, a)}|\tilde{f}_i^j| = 1.\end{equation}
  \begin{equation}\label{32}\min\limits_{x\in\partial B(p_i, 1)}\sum\limits_{j=1}^K|\tilde{f}_i^j(x)|^2> 2\sup\limits_{x\in B(p_i, a)}\sum\limits_{j=1}^K|\tilde{f}_i^j(x)|^2.\end{equation}  Furthermore, for all $j$, \begin{equation}\label{33}\frac{\sup\limits_{x\in B(p_i, 3)}|\tilde{f}_i^j(x)|^2}{\sup\limits_{x\in B(p_i, 2)}|\tilde{f}_i^j(x)|^2}\leq C=C(M).\end{equation} 

Lemma \ref{l21} says \begin{equation}\label{34}|\tilde{f}^j_i|\leq C\end{equation} on $B(p_i, 3)$.
Let $\lambda_i$ be a sequence tending to $\infty$ very slowly.
Let $w_i = \log(u+1)$, where $u$ was defined in proposition \ref{p3}. We assume $u$ is defined on $B(p_i, \lambda_i)$.
Then \begin{equation}\label{35}\sqrt{-1}\partial\overline\partial w_i = \frac{(u+1)\partial\overline\partial u-\partial u\wedge\overline\partial u}{(u+1)^2}.\end{equation}
According to proposition \ref{p3}, since $\lambda_i$ is increasing so slowly, for all large $i$, $w_i$ is strictly plurisubharmonic on $B(p_i, \lambda_i)\backslash B(p_i, 1)$. Furthermore, outside $B(p_i, \frac{1}{10})$, \begin{equation}\label{36}|w_i- \log (r_i^2+1)|<\Phi(\frac{1}{i}).\end{equation} Also, there exists $c>0$ independent of $i$ so that on $B(p_i, 10)\backslash B(p_i, \frac{1}{10})$,\begin{equation}\label{37}\sqrt{-1}\partial\overline\partial w_i\geq c\omega_i.\end{equation}
Now consider a cut off function $\eta_i$ depending only on $r_i$ so that $\eta_i = 1$ on $B(p_i, 2)$, $\eta_i$ has compact support on $B(p_i, 2.5)$.  By using the weight $Cw_i$ ($C$ is large independent of $i$), we can solve the $\overline\partial$ equation $\overline\partial  v_i=\overline\partial (\eta_i\tilde{f}^j_i)$, as on page $803$ of \cite{[L3]}. Then we obtain holomorphic functions $f^j_i$ on $B(p_i, \lambda_i)$ which are close $\tilde{f}^j_i$ on $B(p_i, 2)$. Furthermore, \begin{equation}\label{38}|f^j_i(x)|\leq C(1+d(x, p_i))^{d_0}\end{equation} for some $d_0$ depending only on $M$. By shifting $f^j_i$ by small constants, we may assume \begin{equation}\label{39}f^j_i(p_i) = 0, \sum\limits_{j=1}^k|f^j_i|^2\geq c>0\end{equation} on $\partial B(p_i, 1)$. By passing to subsequence, we assume $f^j_i\to f^j$ uniformly on each compact set of $M_\infty$.

Let $V = \{g\in\mathcal{O}_{d_0}(M)|g(p) = 0\}$.
Let $g_s$ $(s= 1, .., N)$ be an orthonormal basis of $V$ with respect to the $L^2$ inner product on $B(p, 1)$. Assume $(g_1, .., g_N)$ is not proper.
Then there exist a constant $C$ and a sequence $\mu_i\to\infty$ so that \begin{equation}\label{310}\min\limits_{\partial B(p, \mu_i)}\sum |g_s|^2\leq C.\end{equation}

For each $i$, There exists a basis $g^i_s$ of $V$ with \begin{equation}\label{311}\int_{B(p, 1)}g^i_s\overline{g^i_t} = \delta_{st}; \int_{B(p_i, 1)}g^i_s\overline{g^i_t}=\lambda^i_{st}\delta_{st}.\end{equation}  Then \begin{equation}\label{312}\sum\limits_{s}|g_s|^2 = \sum\limits_{s}|g^i_s|^2. \end{equation} As $g^i_s(p) = 0$, $\sup\limits_{B(p, 1)}|dg^i_s|\geq c>0$ for any $i$. Gradient estimate implies that \begin{equation}\label{313}\lambda^i_{ss}>c\mu_i^2.\end{equation}

Now for the sequence $\mu_i\to \infty$, we can find a basis $h^i_1, .., h^i_N$ of $V$ so that $\dashint_{B(p_i, 1)}h^i_j\overline h^i_k = \delta_{jk}$. Mean value inequality and corollary \ref{cor21} imply that we can pass $h^i_j$ to $M_\infty$, after taking subsequence. Say $h^i_j\to h_j$. Then $h_j\in\mathcal{O}_{d_0}(M_\infty)$.
\begin{claim}\label{cl31}
Span$\{f^k\}\subset$ Span$\{h_j\}$.
\end{claim}

\begin{proof}
The argument is the same as claim $6.1$ in \cite{[L3]}, except that we replace three circle theorem by corollary \ref{cor21}. We skip the details.
\end{proof}

Claim \ref{cl31} implies that on $B(p_i, 1)$, $f^k_i$ is almost in the span of $h^i_j$. 
More precisely, \begin{equation}\label{314}\lim\limits_{i\to\infty}\sup_{B(p_i, 1)}|f^j_i(x) - \sum\limits_{s}c^i_{js}h^i_s| = 0\end{equation} for $c^j_{is} = \int_{B(p_i, 1)}f^j_i\overline{h^i_s}$. In particular, $|c^i_{js}|\leq C(M)$. By (\ref{39}), \begin{equation}\label{315}\min_{\partial B(p_i, 1)}\sum\limits_{j=1}^{K}|f^j_i(x)|^2 \geq c>0.\end{equation} By (\ref{313}), \begin{equation}\label{316}|h^i_s|^2 = \frac{|g^i_s|^2}{\lambda^i_{ss}}\leq\frac{|g^i_s|^2}{c\mu_i^2}.\end{equation} Then from (\ref{312}), \begin{equation}\label{317}\min\limits_{\partial B(p, \mu_i)}\sum\limits_{s}|g_s|^2=\min\limits_{\partial B(p_i, 1)}\sum\limits_{s}|g_s|^2=\min\limits_{\partial B(p_i, 1)}\sum\limits_{s}|g^i_s|^2 \geq c\mu_i^2>0.\end{equation} This contradicts (\ref{310}).
\end{proof}

Let $F = (g_1, .., g_N)$.
Proper mapping theorem says the image $F(M)$ is an irreducible analytic subvariety $X$ of $\mathbb{C}^N$.
The preimage of a point of $X$ is a compact subvariety. By corollary \ref{cor22}, we can pick a point on $X$ far away so that the its preimage consists of finitely many points. Therefore, $dim(X) = n$. Combining with proposition \ref{p1} that $dim(\mathcal{O}_d(M))\leq Cd^n$, we find the transcendental dimension of $\mathcal{O}_{d_0}(M)$ over $\mathbb{C}$ is $n$. Let $Y\subset\mathbb{C}^N$ be the affine algebraic variety determined by the integral ring $\mathbb{C}[g_1, .., g_N]$. Then $Y$ has dimension $n$ and $X$ is a subvariety of $Y$.
As $Y$ is irreducible and $dim(X) = n$, $X = Y$.
Thus $X$ is an affine algebraic variety. The function $v=\log(1+|g_1|^2+\cdot\cdot\cdot+|g_N|^2)$ is psh with logarithmic growth on $M$. Now pick a generic point $q\in X$ so that $F^{-1}(q)=\{q_1, .., q_m\}\in M$. We further require that $F$ is a local biholomorphism near $q_1, .., q_m$. Then $v$ is strictly psh near $q_1, .., q_m$. By solving $\overline\partial$ equation with weight function $Cv$ ($C$ is large), we obtain polynomial growth holomorphic functions which separate $q_1, .., q_m$. By adding these functions to $F$, we find that $F$ is generically one to one. Recall $X$ is the image of $F$ which is affine. Let us replace $X$ by the normalization $\overline{X}$. According to proposition \ref{p2}, it suffices to add finitely many polynomial growth holomorphic functions to $F$, in order to achieve $\overline{X}$. We still denote $\overline{X}$ by $X$, for simplicity.

Then $F$ is a biholomorphism from $M$ to $X$ outside the compact subvarieties $Z_1, ..., Z_k$, which are contracted to points $z_1, .., z_k$ on $X$. Let $R(M)$ be the field of polynomial growth holomorphic functions. As the transcendental dimension of $R(M)$ over $\mathbb{C}$ is $n$, 
 $R(M)$ can be generated by $(n+1)$ polynomial growth holomorphic functions. Say $g'_1, .., g'_{n+1}$. By adding these functions to the map $F$, we may assume $g'_1, .., g'_{n+1}$ are regular functions on $X$. Therefore,
any polynomial growth holomorphic functions $f$ on $M$ is a rational function on $X$. Since $X$ is normal, $f$ is regular \cite{[H]}. Thus the ring of polynomial growth holomorphic functions on $M$ can be identified with the affine coordinate ring of $X$. Hence, it is finitely generated.

Now $X$ can be compactified as a projective manifold by attaching some ample divisor $D$ at infinity. Since $F$ is a biholomorphism outside compact set, we can also attach $D$ to $M$ at infinity. The resulting manifold is Moishezon, since it contains algebraic functions $g_1, .., g_N$ which have transcendental dimenion $n$.
Thus $M$ is biholomorphic to a Zariski open set of a Moishezon manifold.

\bigskip

\section{\bf{Proof of theorem \ref{thm1}, part II}}
Let $C(M)$ be the unique tangent cone of $M$ at infinity. 
Let $D = \{\gamma\geq 0|$ there exists a nonzero harmonic function $f$ on $C(M)$ so that $f(\lambda x) = \lambda^\gamma f(x)\}$, where $\lambda$ is the homothety map on the metric cone $C(M)$. Let us call $D$ the spectrum of harmonic functions on $C(M)$.
Assume $\alpha$ is not contained in the spectrum of harmonic functions on $C(M)$. 
This is always possible, since the values in $D$ are discrete (they are related to the spectrum of harmonic functions on the cross section).

We need a result of Donaldson-Sun (proposition $3.23$ of \cite{[DS2]}):
\begin{lemma}\label{l41}
Let $(M, p)$ be a complete noncompact K\"ahler manifold with nonnegative Ricci curvature and maximal volume growth. Assume $\alpha$ is not contained in the spectrum of harmonic functions on any tangent cone of $M$. Then for $R$ large enough, for any holomorphic function $f$ on $B(p, 4R)$, if $\frac{I(f, 2R)}{I(f, R)}\leq 2^\alpha$, then $\frac{I(f, R)}{I(f, \frac{R}{2})}<2^\alpha$, where $I(f, R) = \dashint_{B(p, R)}|f|^2$.
\end{lemma}

\begin{prop}\label{p41}
Under the assumption of theorem \ref{thm1}, part II, there exist holomorphic functions of polynomial growth $g_1, ..., g_N$ so that the map $(g_1, .., g_N)$ is proper.\end{prop}
 \begin{proof}The argument is almost the same as proposition \ref{p31}. We only indicate the minor differences. 
 
 \medskip
 
 1. In (\ref{38}), by increasing $d_0$ if necessary, we may assume $d_0\notin D$, the spectrum of harmonic functions on $C(M)$. Recall between (\ref{313}) and claim \ref{cl31}, when we pass $h^i_j$ to $M_\infty$, corollary \ref{cor21} is used. Now we replace corollary \ref{cor21} by lemma \ref{l41}.
 
 \medskip
 
 2. We still have $dim(\mathcal{O}_d(M))\leq Cd^n$, by corollary \ref{cor24}.
 
 \medskip
 
 3. In the proof of claim $6.1$ of \cite{[L3]}, the three circle theorem used between ($6.90$) and ($6.91$) should be replaced by lemma \ref{l41}.

 \end{proof}
 
 By exactly the same argument as in part I, we can prove that $M$ is biholomorphic to a Zariski open set of a Moishezon manifold.
 
 \bigskip

Now we prove corollary \ref{cor1}.
  In the Ricci flat case, we can apply the result of Colding-Minicozzi \cite{[CM]} to see the uniqueness of the tangent cone. Hence, part II of theorem \ref{thm1} can be applied. Alternatively, we can also adapt an argument of Donaldson-Sun. Basically from proposition $2.21$ of \cite{[DS2]}, we see that the holomorphic spectrum on tangent cone must be algebraic numbers. Hence lemma \ref{l41} can still be applied. The first statement of corollary \ref{cor1} is a consequence of theorem \ref{thm1}, part II.

The degeneration argument follows from \cite{[DS2]} and \cite{[CSW]}. For reader's convenience, let us give a sketch. Let $V_j = \mathcal{O}_{d_j}(M)$, where $0=d_0<d_1<d_2<..$ and for any $\alpha$ strictly between $d_j$ and $d_{j+1}$, $\mathcal{O}_{\alpha}(M) = \mathcal{O}_{d_j}(M)$. 
Pick basis $f_1, .., f_{m_1}$ of $V_1$. Inductively, we add basis $f_{m_{j-1}+1}, .., f_{m_j}$ to $V_j$.  Let $W_j$ be the span of $\{f_{m_{j-1}+1}, ..., f_{m_j}\}$.

\begin{lemma}\label{l42}
$dim(\mathcal{O}_d(M)) = dim(\mathcal{O}_d(C(M))$ for any $d$.
\end{lemma}
\begin{proof}
The proof is the same as in proposition $3.26$ of \cite{[DS2]}. 
\end{proof}

\begin{lemma}\label{l43}$\oplus_{k\geq 0}\mathcal{O}_{d_{k}}(M)\slash \mathcal{O}_{d_{k-1}}(M)$ is finitely generated.
\end{lemma}
\begin{proof}  The argument is the same as in proposition $3.14$ and lemma $3.15$ of \cite{[DS2]}. 
\end{proof}

Choose $j$ large so that $\oplus_{k\geq 0}\mathcal{O}_{d_{k}}(M)\slash \mathcal{O}_{d_{k-1}}(M)$ is generated by $\oplus_{j\geq k\geq 0}\mathcal{O}_{d_{k}}(M)\slash \mathcal{O}_{d_{k-1}}(M)$.
Furthermore, we require that $F=(f_1, .., f_{m_1}, ..., f_{m_{j-1}+1}, .., f_{m_j})$ is a biholomorphism outside compact set of $M$ to its image in $\mathbb{C}^N$, where $N = m_j$.
Let $X= F(M)$. As we see before, $X$ is a normal affine algebraic variety. $F$ just contracts finitely many compact subvarieties to points on $X$.

Recall definition \ref{d1}.
Consider a degeneration of $X$ in $\mathbb{C}^N$ given by $\sigma_t(f_k) = t^{deg(f_k)}f_k$, where $t\to 0^+$ as positive real numbers. Then $X$ degenerates to the affine algebraic variety $W = $ Spec $\oplus_{k\geq 0}\mathcal{O}_{d_{k}}(M)\slash \mathcal{O}_{d_{k-1}}(M)$.

Take a sequence $r_i\to \infty$. Define $(M_i, p_i, d_i) = (M, p, \frac{d}{r_i})$. Assume $(M_i, p_i, d_i)$ converges in the pointed Gromov-Hausdorff sense to a tangent cone $C(M)=(M_\infty, p_\infty, d_\infty)$ at infinity. 

\begin{prop}\label{p42}
For any $g\in W_j, \lim\limits_{r\to\infty}\frac{I(g, 2r)}{I(g, r)} = 2^{d_j}$.
\end{prop}
\begin{proof}Let $\epsilon>0$ be a small number.
Since $g\in W_j$, we can find a sequence $R_{i, \epsilon}\to\infty$ so that $\frac{I(g, 2R_{i, \epsilon})}{I(g, R_{i, \epsilon})}<2^{d_j+\epsilon}$. 
As $g\not\in \mathcal{O}_{d_j-\epsilon}(M)$, we can find a sequence $R'_{i, \epsilon}\to\infty$ so that $\frac{I(g, 2R'_{i, \epsilon})}{I(g, R'_{i, \epsilon})}>2^{d_j-\epsilon}$.
By the discreteness of the spectrum of the tangent cone,  $d_j-\epsilon, d_j+\epsilon$ are not on the spectrum of harmonic functions of $C(M)$. By letting $\epsilon\to 0$, we obtain the proposition from lemma \ref{l41}.

\end{proof}

For any $k\in\mathbb{N}$, let $f^i_{m_k+1}, .., f^i_{m_{k+1}}$ be a basis of $W_k$ so that $\dashint_{B(p_i, 1)} f^i_s\overline{f^i_l} = \delta_{sl}$, where $m_k+1\leq s, l\leq m_{k+1}$. Proposition \ref{p42} and lemma $3.6$ of \cite{[DS2]} imply that after passing to subsequence, $f^i_s$ converges, uniformly in each compact set of $C(M)$, to a homogeneous polynomial growth holomorphic function $f^\infty_s$ of degree $d_k$. Furthermore, since any homogeneous holomorphic function with different degree are $L^2$ orthogonal on $B(p_\infty, 1)$, we conclude for any $1\leq a, b\leq m_{j}$, $\dashint_{B(p_\infty, 1)}f^\infty_a\overline{f^\infty_b} = \delta_{ab}$.
Lemma \ref{l42} implies that $f^\infty_1, .., f^\infty_{m_j}$ is a basis of $\mathcal{O}_{d_j}(C(M))$.

The tangent cone of $C(M)$ admits a natural $\mathbb{T}$ action generated by $r\frac{\partial}{\partial r}$ and $\sqrt{-1}rJ\frac{\partial}{\partial r}$. For any homogeneous polynomial growth holomorphic function $f$ of degree $a$, $r\frac{\partial}{\partial r}f = af$, $rJ\frac{\partial}{\partial r}f = \sqrt{-1}af$.
Spec $C(M)$ has a natural grading by the degree of homogeneous polynomial growth holomorphic functions. Note by lemma \ref{l42}, the grading is the same as Spec $W$.
Thus there is also a natural $\mathbb{T}$ action on $W$, with the same Hilbert function as $C(M)$.

According to \cite{[DS2]}\cite{[HS]}, there is a multi-graded Hilbert scheme \textbf{Hilb}, which is a projective scheme parametrizing polarized affine schemes in $\mathbb{C}^N$ invariant under the $\mathbb{T}$ action with fixed Hilbert function. We can consider the embedding of $W$ and $C(M)$ to $\mathbb{C}^N$ by $([f^i_1], ..., [f^i_{m_j}])$ and $(f^\infty_1, .., f^\infty_{m_j})$, where $[\cdot]$ means the quotient in $\mathcal{O}_{d_{k}}(M)/\mathcal{O}_{d_{k-1}}(M)$. These define points $[W_i]$ and $[C(M)]$ in \textbf{Hilb}.

Let $G$ be the linear transformations of $\mathbb{C}^N$ that commutes with $\mathbb{T}$ action. Let $K = G\cap U(N)$, where $U(N)$ is the unitary group.
With the same proof as in proposition $3.16$ of \cite{[DS2]}, we obtain 
\begin{prop}\label{p43}
$[W_i]$ converges to $[C(M)]$, up to $K$ action. 
\end{prop}

By Matsushima's theorem (proposition $4.9$ of \cite{[DS2]}), $Aut(C(M))$ preserving $\mathbb{T}$ is reductive. By general theory, there exists a degeneration from $W$ to $C(M)$. 
This completes the proof of corollary \ref{cor1}.
\begin{remark}
We can apply the Luna slice theorem as in \cite{[DS2]} to prove that all tangent cones of $M$ at infinity are isomorphic as affine algebraic varieties. By the uniqueness of the K\"ahler-Einstein metric, we find that the tangent cone is unique. This recovers Colding-Minicozzi's result \cite{[CM]} in the K\"ahler case.
\end{remark}

\section{\bf{Proof of theorem \ref{thm1}, part III}} 

In this section, we assume the K\"ahler manifold $M^n (n\geq 2)$ has positive Ricci curvature, maximal volume growth and $|Rm|\leq\frac{C}{r^2}$. The only difference from \cite{[Mo]} is the absence of $\int_MRic^n<\infty$. We follow Mok's approach in \cite{[Mo]}. In fact, the argument is almost the same as in \cite{[Mo]}. For completeness, we shall include some details, with emphasis on the difference.

The main strategy is to consider the embedding by pluri-anticanonical sections with polynomial growth. 
\begin{definition}
Let $K$ be the canonical line bundle of $M$. Define $P^{q}_d(M) = \{f\in \Gamma(M, K^{-q})| |f|\leq C(1+r)^d\}$ for some $C>0$. Set $P(M)=\cup_{q>0, d>0}P^{q}_d(M)$.\end{definition}

Following Mok \cite{[Mo]}, let us divide the proof into several steps:

\bigskip

\emph{Step 1: Uniform multiplicity estimate for pluri-anticanonical system}

\begin{prop}\label{p51}
Let $(M, p)$ be a complete noncompact K\"ahler manifold with nonnegative Ricci curvature and maximal volume growth. Assume further that $|Rm|\leq \frac{C}{r^2}$. For any nonzero $s\in P^q_d(M)$, let $V$ be the zero divisor of $s$. Then for any $x$ which is not on any compact subvariety of positive dimension, $Mul_x(V) \leq C(d+q)$ for some constant $C$ independent of $s$ and $x$.
\end{prop}
\begin{remark}
Compare theorem $2.3$ of \cite{[Mo]}, where it was proved that the multiplicity estimate holds except for $x$ on a disjoint union of compact subvarieties depending on $s$. Notice that the union of these compact subvarieties could be zero dimensional and noncompact.
\end{remark}
\begin{proof}
This is just a combination of proposition \ref{p28} and proposition \ref{p29}.
\end{proof}

\bigskip

\emph{Step 2:  Quasi-embedding into a projective variety}

\medskip

Let $R(M)$ be the degree zero part (rational functions) of the quotient field of $P(M)$.
Similar as proposition $3.2$ of \cite{[Mo]} , we have 
\begin{prop}\label{p52}
$R(M)$ has transcendental dimension $n$ over $\mathbb{C}$. Moreover, $R(M)$ separates points and tangents on $M$. In particular, $R(M)$ a finite extension over $\mathbb{C}(f_1, .., f_n)$ for some algebraically independent rational functions $f_1, .., f_n\in R(M)$.
\end{prop}
\begin{remark}
In \cite{[Mo]}, the argument requires theorem $2.2$ of \cite{[Mo]}, which involves the finiteness of $\int_M Ric^n$.
\end{remark}
\begin{proof}
By the argument on page $383$ of \cite{[Mo]}, for large $q$, we can find nontrivial $L^2$ holomorphic sections of the pluri-anticanonical bundle $K^{-q}$.  Furthermore, standard $\overline\partial$ estimates imply that these sections separates points and tangents of $M$. Let us verify that these $L^2$ sections are in $P(M)$. Say $s\in \Gamma_{L^2}(M, K^{-q})$. Then $s$ satisfies \begin{equation}\label{51}\frac{\sqrt{-1}}{2\pi}\partial\overline\partial\log |s|^2 = [D]-qRic,\end{equation} where $D$ is the divisor of $s$. By taking the trace, we obtain that \begin{equation}\label{52}\frac{1}{2\pi}\Delta(2q\rho'+\log |s|^2)\geq 0,\end{equation} where $\rho'$ is in proposition \ref{p28}. Therefore, $|s|^2e^{2q\rho'}$ is a subharmonic function on $M$. Note \begin{equation}\label{53}\int_{B(p, 2r)} |s|^2e^{2q\rho'}\leq \sup\limits_{B(p, 2r)}e^{2q\rho'}\int_{B(p, 2r)}|s|^2\leq C(1+r)^{qC}.\end{equation} Mean value inequality and proposition \ref{p28} imply that $s\in P(M)$. Now assume $s_1, .., s_{n+1}\in P(M)$ and they are algebraically independent. By taking some power if necessary, we may assume $s_1, .., s_{n+1}\in P^{q_0}_{d_0}$. According to proposition \ref{p28} and corollary \ref{cor24}, $dim(P^{kq_0}_{kd_0}(M))\leq Ck^n$, where $C$ is independent of $k$. A simple linear algebra argument yields a contradiction to the assumption that $s_1, ..., s_{n+1}$ are algebraically independent.
\end{proof}

By a meromorphic map of $M$ to $\mathbb{P}^N$, we mean a holomorphic map into $\mathbb{P}^N$ on $M-Q$ for some subvariety $Q$ of codimension $\geq 2$. By a quasi-embedding of $M$ into $\mathbb{P}^N$, we mean a meromorphic map $F$ for which there exists a subvariety $Q$ of $M$ such that $F_{M-Q}$ is a holomorphic embedding into $\mathbb{P}^N$. 

According to proposition \ref{p52}, we can assume $R(M) = \mathbb{C}(f_1, .., f_n, g)$, where $g$ is algebraic over $\mathbb{C}(f_1, .., f_n)$ and $f_1, .., f_n\in R(M)$ are algebraically independent. We may assume $f_1, .., f_n, g$ have common denominator $s_0\in P(M)$. Say $f_i = \frac{s_i}{s_0}$, $g = \frac{u}{s_0}$. Let $Q$ be the common zero of $s_1, .., s_n, u$. Then $F = [s_0:s_1:\cdot\cdot\cdot:s_n: u]$ defines a holomorphic map from $M-Q_0$ to $\mathbb{P}^{n+1}$, where $Q_0$ is a subvariety of $Q$ of codimension at least $2$ in $M$.

\begin{prop}\label{p53}
The meromorphic map $F$: $M\to\mathbb{P}^{n+1}$ is a quasi-embedding into some irreducible hypersurface $Z$ of $\mathbb{P}^{n+1}$.\end{prop}

The proof is the same as proposition $3.3$ of \cite{[Mo]}.

\bigskip

\emph{Step 3: Almost surjectivity of quasi-embedding}

\medskip

The rough idea is this: given any $\xi\in  Z-S-F(M-W)$ where $S$ is certain proper subvariety of $Z$, construct $g\in P(M)$ which extends to a meromorphic section on $Z$ and have pole at $\xi$. Then prove $Z-S-F(M-W)$ is given by union of finitely many divisors. The main tool is Skoda's $L^2$ estimate.

\medskip
Let us summarize some construction in section $4$ of \cite{[Mo]}.
Let $F= [s_0:s_1:\cdot\cdot\cdot:s_n:u]$ be as in \cite{[Mo]}. Recall $f_i = \frac{s_i}{s_0}$. Let $W$ be the union of zero set $W_0$ of $s_0$ and the branching locus of the holomorphic map $\rho_0(x) = (f_1(x), .., f_n(x))$ defined on $M-W_0$. Then $M-W$ is a Stein manifold and $\rho=\rho_0|_{M-W}$ realizes $M-W$ as a Riemann domain of holomorphy over $\mathbb{C}^n$. It turns out that there is a positive integer $N$ so that for any $z\in\mathbb{C}^n$, $\rho^{-1}(z)$ contains at most $N$ points. On the Riemann domain of holomorphy $\rho: M-W\to\mathbb{C}^n$, let $\delta$ denote the distance to the boundary as on page $386$ of \cite{[Mo]}.
The following proposition is the key ingredient of section $4$ of \cite{[Mo]}. Let us provide a variant proof.
\begin{prop}\label{p54}
There exists $s\in P(M)$ so that $\delta(x) \geq c_1|s|^{a}|s_0|^{b}(1+r)^{-c}$, where $a, b, c, c_1$ are some positive constants.
\end{prop}
\begin{proof}
The following is lemma $4.3$ of \cite{[Mo]}.
\begin{lemma}\label{l51}
For any $s\in P(M)$, $||\nabla s||$ has polynomial growth.
\end{lemma}
\begin{proof}
From the assumption of $M$, it is easy to see that the injectivity radius has a lower bound $c>0$. As $|Rm|\leq C$, at each point $x\in M$, there exists a $C^{1, \alpha}$ holomorphic chart $(z_1, .., z_n)$ containing $B(x, r_0)$ for some fixed $r_0>0$. 
In particular, the Christoffel symbol is $C^\alpha$ continuous. Since $0<c I\leq g_{i\overline{j}}\leq C I$, locally we can write $s = s_x(\frac{\partial}{\partial z_1}\wedge\cdot\cdot\wedge \frac{\partial}{\partial z_n})^q$, where $s_x$ is a holomorphic function near $x$ with polynomial bound of $r(x)$. Then $||\nabla s|| = ||(\nabla s_x) (\frac{\partial}{\partial z_1}\wedge\cdot\cdot\wedge \frac{\partial}{\partial z_n})^q + s_x \nabla(\frac{\partial}{\partial z_1}\wedge\cdot\cdot\wedge \frac{\partial}{\partial z_n})^q||$ has polynomial growth.

\end{proof}
Notice $s_0^2df_k = s_0\nabla s_k-s_k\nabla s_0$. Then according to lemma \ref{l51}, \begin{equation}\label{54}|df_k|\leq \frac{C(r+1)^d}{|s_0|^2}.\end{equation} We also obtain that $s := s_0^{2n}df_1\wedge\cdot\cdot\wedge df_n\in P(M)$.  Since $||\nabla s||$ has polynomial bound, if $|s(x)|\neq 0$,
there exists $c>0, d>0$ so that on $B(x, \frac{c}{(1+r)^d}|s(x)|)$, \begin{equation}\label{55}|s|\geq \frac{1}{2}|s(x)|.\end{equation} Similarly, on $B(x, \frac{c}{(1+r)^d}|s_0(x)|)$, \begin{equation}\label{56}|s_0|\geq \frac{1}{2}|s_0(x)|.\end{equation} Let $\mu = min(\frac{c}{(1+r)^d}|s(x)|, \frac{c}{(1+r)^d}|s_0(x)|, r_0)$ (recall $r_0$ is the size of the holomorphic chart in lemma \ref{l51}).
This implies that on $B(x, \mu)$, in terms of the holomorphic chart $(z_1, .., z_n)$ in lemma \ref{l51}, \begin{equation}\label{57}|\frac{\partial f_k}{\partial z_i}|\leq \frac{C(r+1)^d}{|s_0(x)|^2}, |\det(\frac{\partial f_k}{\partial z_i})|\geq \frac{c|s(x)|}{|s_0(x)|^{2n}}.\end{equation}

Let $q_x = (f_1(x), .., f_n(x))$.
By standard Cauchy estimate and integration (see also the ODE argument in section $4$ of \cite{[Mo]}.), we conclude $(f_1, .., f_n)|_{B(x, \mu)}$ is a holomorphic chart and the image contains $B_{\mathbb{C}^n}(q_x, c_1|s|^a|s_0|^br^{-c})$ for some $a, b, c, c_1>0$. This concludes the proof of proposition \ref{p54}.

\end{proof}

One can apply Skoda estimate as in \cite{[Mo]}. With exactly the same argument, we have 
\begin{prop}\label{p55}
There exists a subvariety $T$ of $Z$ such that $F(M-W) = Z-T$.
\end{prop}

\bigskip

\emph{Step 4: Completion of the proof of theorem \ref{thm1}, part III}

\medskip

 We shall add more polynomial growth sections to $F$ so that $F^i=[s_0: \cdot\cdot : s_{N_i}]: M\to\mathbb{P}^{N_i}$ is a holomorphic embedding onto its image which is quasi-projective. It suffices to solve two problems:
 
 \medskip
 
 (a) Prove that the base locus of $s_0, .., s_{N_i}$ is empty.
 
 \medskip
 
 (b) Prove that there is no branch points for $F^i$.
 
 \medskip
 
 \begin{prop}\label{p56}
There exists a quasi-embedding $F': M\to Z'\subset \mathbb{P}^N$ into a normal projective variety such that for some disjoint union $G$ of discrete points of $M$,  $F'|_{M-G}$ is a biholomorphism from $M-G$ onto a Zariski open subset of $Z'$.
\end{prop}

The argument is exactly the same as proposition $5.2$ of \cite{[Mo]}.  Notice that by corollary \ref{cor22}, any compact subvariety outside a compact set must be zero dimensional.
Also the uniform multiplicity estimate, proposition \ref{p51} plays an essential role.

\medskip

We claim that $F': M\to Z'$ is holomorphic except finite many points. Let $q\in G$ be so that $F'$ does not extend holomorphically through $q$. As remarked on page $395$ of \cite{[Mo]},
there exists $U\ni q$ so that $U-q$ is biholomorphic to $V-K$ where $V$ is an open set of $Z'$ and $K$ is a compact subvariety of $V$ of positive dimension. By a Mayer-Vietories sequence argument, we find that such $K$ gives a nontrivial contribution to $H_2(Z', \mathbb{R})$.  For different $q$ in $G$, the contributions to $H_2(Z', \mathbb{R})$ are linear independent. As $Z'$ has finite topological type, the number of such $q\in G$ must be finite.

Then we just add finitely polynomial growth sections to desingularize and separate these finitely many points. Theorem $1$, part III is concluded.
\begin{remark}
In \cite{[Mo]}, there is a section on Bezout estimate, where the extra condition $\int Ric^n<\infty$ was used. This section is unnecessary for us, due the fact that
there is no compact subvariety of positive dimension outside certain compact set (corollary \ref{cor22}).
\end{remark}

\end{document}